\documentclass[12pt]{article}

\usepackage{fullpage}
\usepackage{fancyhdr}
\usepackage{amsmath}
\usepackage{amssymb}

\def\1{\underline{1}}
\def\R{\Bbb R}

\def\Z{\mathbb Z}

\def\C{\mathbb C}

\newtheorem{theorem}{Theorem}

\newtheorem{proposition}{Proposition}

\newenvironment{definition}
{\smallskip\noindent{\bf Definition\/}:}{\smallskip\par}

\newenvironment{remark}
{\smallskip\noindent{\bf Remark\/}.}{\smallskip\par}

\newenvironment{proof}
{\noindent{\bf Proof\/}.}{{ $\Box$}\smallskip\par}

\title{On an equivariant analogue of the monodromy zeta function
\footnote{Math. Subject Class. 32S05, 57R91, 58K10. 
Keywords: finite group actions, zeta function of a map, monodromy.
} }

\author{S.M.~Gusein-Zade
\thanks{
Partially supported by the the Russian government grant 11.G34.31.0005, RFBR--10-01-00678,
NSh--4850.2012.1 and Simons--IUM fellowship.
Address: 
Moscow State University, Faculty of Mechanics and Mathematics,
GSP-1, Leninskie Gory, Moscow, 119991, Russia.
E-mail: sabir\symbol{'100}mccme.ru
}
}

\date{}
\begin{document}
\def\eps{\varepsilon}

\maketitle

\begin{abstract}
We offer an equivariant analogue of the monodromy zeta function of a germ
invariant with respect to an action of finite group $G$ as an element
of the Grothendieck ring of finite $(\Z\times G)$-sets. We formulate 
equivariant analogues of the Sebastiani--Thom theorem and of the A'Campo formula.
\end{abstract}

This paper (in a slightly shorter form) was essentially written in February 2008,
when it was lost with an external hard disk in Madrid.
For a rather long time there was no mood to restore it.
However recently some ideas from it were used, in particular, in \cite{RMC}, \cite{EG}.
This urged to return to the work.
The most part of the paper was written when the author used the hospitality
of the Max Plank Institute for Mathematics in Bonn.

A number of invariants of singularities of analytic spaces and of germ of
analytic functions on them have equivariant versions for singularities
with an action of a finite group $G$ and for $G$-invariant (or $G$-equivariant)
function germs on them. For example, in \cite{Wall}, as an equivariant analogue
of the Milnor number of a $G$-invariant function germ on the space $(\C^n,0)$ (with
an action of the group $G$) one considers a certain element of the ring $R(G)$
of (virtual) representations of the group $G$.
An important invariant of a germ of an analytic function on a germ of an analytic space
(in general, singular) is the classical monodromy zeta function of the germ.

Let $f:(V,0)\to (\C,0)$ be a germ of a complex analytic function on a germ
$(V,0)\subset(\C^N,0)$ of an analytic space (singular, in general). For a small enough
$\delta>0$ the restriction of the function $f$ to the intersection $V\cap B_\delta(0)$
of the space $V$ with the (closed) ball $B_\delta(0)$ of radius $\delta$ with the centre
at the origin in $\C^N$ is a locally trivial fibration over a punctured neighbourhood
of zero in $\C$ ({\em the Milnor fibration} of the germ $f$; see \cite{Le}). The fibre
$M_f=f^{-1}(\eps)\cap B_\delta(0)$ of this fibration ($0<\vert\eps\vert<<\delta$ are small enough)
is called {\em the Milnor fibre} of the germ $f$. A monodromy transformation $\Gamma_f:M_f\to M_f$
of the Milnor fibration of the germ $f$ corresponding to the loop $\eps\cdot\exp{(2\pi i\tau)}$
($\tau\in[0,1]$) is called {\em a (classical) monodromy transformation} of the germ $f$.
(It is well defined up to isotopy.)
The action of a monodromy transformation in the (co)homology groups of the Milnor fibre
is called {\em the (classical) monodromy operator} of the germ $f$.
The zeta function of a transformation $\varphi:X\to X$ (of a good enough space $X$, say,
of a union of cells in a finite CW-complex or of a real semi-analytic space) is the rational
function
$$
\zeta_\varphi(t):=
\prod\limits_{q\ge0}\left\{\det\left(id-t\varphi^*_{\vert H^q_c(X;\R)}\right)\right\}^{(-1)^q}\,,
$$
where $\varphi^*$ is the action of the transformation $\varphi$ in the cohomology groups
$H^q_c(X;\R)$ with compact support. The degree of the zeta function $\zeta_\varphi(t)$
(the degree of the numerator minus the degree of the denominator) is equal to the Euler
characteristic $\chi(X)$ of the space $X$ (also defined through the cohomologies with compact
support). The classical monodromy zeta function $\zeta_f(t):=\zeta_{\Gamma_f}(t)$ of a germ $f$
can be represented in the form $\prod\limits_{m\ge1} (1-t^m)^{s_m}$, where $s_m$ are integers,
finitely many of which are different from zero. This follows, for example, from an appropriate
version of the A'Campo formula \cite{A'Campo}.

\begin{remark}
 The use of the cohomology groups with compact support in the definitions is connected with the fact
 that in this case the Euler characteristic possesses the property of additivity with respect to
 a partition into non-intersecting subspaces, and the zeta function of a transformation possesses
 the corresponding property of multiplicativity, an equivariant version of which is formulated below
 in Proposition~\ref{multi}.
\end{remark}

The zeta function $\zeta_\varphi(t)$ of a transformation $\varphi:X\to X$ can be expressed through the
Lefschetz numbers $L(\varphi^k)$ of the powers $\varphi^k$ of the transformation $\varphi$, $k\ge 1$. 
Let us define integers $r_i$, $i=1, 2\ldots,$ by the recurrent equations
\begin{equation}\label{lefsc}
L(\varphi^m)=\sum_{i|m}r_i\,.
\end{equation}
The number $r_m$ ``counts'' the points $x\in X$, whose $\varphi$-orbits have order $m$
(i.e., $\varphi^m(x)=x$,  $\varphi^{i}(x)\ne x$ for $0<i<m$). Therefore $r_m$ is divisible by $m$:
$r_m=ms_m$. Then
\begin{equation*}
\zeta_\varphi(t)=\prod_{m\geq 1} (1-t^m)^{s_m}.
\end{equation*}
(In this equation for the zeta function of an arbitrary transformation, in contrast to the formula
for the zeta function of a classical monodromy transformation, one may have infinitely many
exponents $s_m$ different from zero.

Assume that, on the germ $(V,0)$, one has an (analytic) action of a finite group $G$ and that the
germ $f$ is invariant with respect to this action: $f(gx)=f(x)$ for any $g\in G$.
The Milnor fibre $M_f$ of the germ $f$ is a $G$-space, i.e. a space with an action of the group $G$,
the cohomology groups $H^q_c(M_f)$ are spaces of $G$-representations ($G$-modules). Let us denote by
$[H^q(M_f)]$ the class of the $G$-module $H^q(M_f)$ in the ring $R(G)$ of (virtual) representations
of the group $G$. The equivariant Euler characteristic of the Milnor fibre $M_f$ of the germ $f$
is usually defined as $\sum\limits_{q}(-1)^q[H^q(M_f)]\in R(G)$
(see, e.g., \cite{Wall}). In this setting a natural generalization of the notion of the monodromy
zeta function of a germ to the equivariant case is absent. We offer a possible generalization and
describe some of its properties.

First, let us somewhat modify the definition of the (equivariant) Euler characteristic.
For convenience we will consider it on the category of real semi-analytic spaces. (The Milnor fibre,
of course,
is such a space.) The Grothendieck group $K_0(\Omega)$ of an algebra $\Omega$ of spaces or of sets
(possibly with additional structures) is the Abelian group generated by the isomorphism classes $[X]$
of the spaces or of the sets from $\Omega$ modulo the relations $[X]=[Y]+[X\setminus Y]$ for
$Y\subset X$. Usually the Grothendieck group of an algebra of spaces or of sets is a ring with the
multiplication defined by the Cartesian product. The Euler characteristic may be regarded as a
homomorphism from the Grothendieck ring of real semi-analytic spaces (or of sets which are unions of
cells in finite CW-complexes) into the ring $\Z$ of integers. In its turn, the ring $\Z$
in a natural way can be identified with the Grothendieck ring $K_0(\mbox{f.s.})$ of finite sets.

In a similar way the equivariant Euler characteristic can be defined as a homomorphism from 
the Grothendieck ring of real semi-analytic spaces with actions of a (finite) group $G$ to
the Grothendieck ring $K_0(\mbox{f.}G-\mbox{s.})$ of finite $G$-sets, called also the Burnside ring
of the group $G$: see, e.g., \cite{TtD}, \cite{LR}. (A finite $G$ set is a finite set with an action
of the group $G$. The Grothendieck ring $K_0(\mbox{f.}G-\mbox{s.})$, as a free Abelian group, is
generated by irreducible finite $G$ sets. Such $G$-sets are in one-to-one correspondence with the
conjugacy classes ${\mathfrak{H}}\in consub(G)$ of subgroups of the group $G$: to a subgroup
$H\subset G$ ($H\in {\mathfrak{H}}$) there corresponds the class of the $G$-set $G/H$.) If $X$ is a
real semi-analytic space with an action of a finite group $G$, then there exists a representation of
$X$ as a union of cells in a finite CW-complex (and even of simplices in a finite simplicial complex)
such that the action of the elements of the group $G$ is coordinated with the partition of $X$
into cells and, moreover, if $g\sigma=\sigma$ for a cell $\sigma$, then $g_{\vert\sigma}=id$.
The set $C_n$ of cells of dimension $n$ is a finite $G$-set. The equivariant Euler characteristic
$\chi_G(X)$ of the $G$-space $X$ can be defined as the alternating sum
$\sum\limits_{n\ge 0}(-1)^n [C_n]\in K_0(f.G-s.)$ of the classes of the $G$-sets $C_n$.
An equivalent definition of the equivariant Euler characteristic $\chi_G(X)$ (which, in particular,
shows that it is well defined) is the following one. For a conjugacy class ${\mathfrak{H}}\in consub(G)$
of subgroups of the group $G$, let $X^{({\mathfrak{H}})} := \{x \in X : G_x \in {\mathfrak{H}}\}$,
where $G_x :=\{g \in G : gx = x\}$ is the isotropy subgroup of the point $x$. Then
\begin{equation}\label{chi_G}
\chi_G(X) = 
\sum\limits_{\mathfrak{H}} \chi(X^{(\mathfrak{H})}/G)[G/H] =
\sum\limits_{\mathfrak{H}}(\chi(X^{(\mathfrak{H})})|H|/|G|)[G/H]\,,
\end{equation}
where the summation is over all the classes ${\mathfrak{H}}\in consub(G)$,
$H$ is a representative of the conjugacy class ${\mathfrak{H}}$.

Let $K_0(\mbox{perm.})$ be the Grothendieck ring of permutations on finite sets.
It can be regarded as the Grothendieck ring of finite $\Z$-sets. 
As an Abelian group, $K_0(perm.)$ is the free Abelian group generated by the cyclic permutations.
For a permutation, as a map from a finite set (zero-dimensional topological space) into itself,
there is defined its zeta-function. The zeta function of the cyclic permutation on $m$ elements
is equal to $(1-t^m)$. To a virtual permutation $(A_1,\sigma_1)-(A_2,\sigma_2)$, let us associate
its zeta function defined as the ratio $\zeta_{\sigma_1}(t)/\zeta_{\sigma_2}(t)$ of the zeta functions
of the permutations $\sigma_i$. It is not difficult to see that this way one defines a one-to-one
map from the Grothendieck ring $K_0(perm.)$ onto the set of rational functions of the form
$\prod (1-t^m)^{s_m}$ with a finite number of the exponents $s_m$ different from zero.
(This defines a structure of a ring on the set of the latter ones. Moreover the ``sum'' of two
functions of the indicated form is their product. This should be kept in mind for understanding
why Proposition~\ref{multi} is mentioned above as an equivariant analogue of the multiplicativity
property of the zeta function of a map.) Thus the zeta function of a function germ can be regarded
as an element of the Grothendieck ring of permutations $K_0(\mbox{perm.})$.
(The zeta function of a germ as an element of the ring $K_0(\mbox{perm.})$
can be defined as the virtual permutation whose ``usual'' zeta function coincides with the ``usual''
zeta function of the germ. Another definition:
this is the virtual permutation, Lefschetz numbers of the powers of which coincide with the Lefschetz
numbers of the powers of the classical monodromy transformation.)
In a similar way the zeta function of a transformation $\varphi:X\to X$ can be regarded as an
element of the Grothendieck ring $K_0(\mbox{l.f.perm.})$ of ``locally finite'' permutations.
(The details can be found below in the equivariant setting.)

An equivariant Lefschetz number $L^G(\varphi)$ of a ($G$-equivariant)
map of a $G$-space $X$ into itself was defined in \cite{LR} as an element of the
Grothendieck ring $K_0(\mbox{f.}G-\mbox{s.})$. It is defined for maps of finite $G$-CW-compexes
and also of the complements to $G$-subcomplexes in finite $G$-CW-complexes. If $X$ is a 
closed $G$-manifold
then $L^G(\varphi)$ can be defined as a (finite) $G$-set of fixed points of a generic $G$-equivariant
perturbation of the map $\varphi$ whose points are taken with appropriate signs.
If $X$ is a compact $G$-manifold with boundary, one has the same statement if one requires
that the perturbation under consideration has no fixed points on the boundary $\partial X$.
then $L^G(\varphi)$ can be defined as a (finite) $G$-set of fixed points of a generic $G$-equivariant
perturbation of the map $\varphi$ whose points are taken with appropriate signs.
(In \cite{GLM2}, one considers another equivariant analogue of the Lefschetz number
adapted to the problem solved there.)

A $G$-permutation on a $G$-set $A$ (finite or infinite one) is a one-to-one $G$-equivariant
map $\sigma$ of the set $A$ into itself or, what is the same, a structure of a $(\Z\times G)$-set
on $A$. 
A $G$-permutation $\sigma$ is {\em finite}, if the corresponding $G$-set $A$ is finite.
We shall say that a $G$-permutation $\sigma$ is {\em locally finite}, if the $\Z$-orbit
of each point is finite and, for any natural $M$, the set of points whose $\Z$-orbits have order
not exceeding $M$ is finite. (A power of a locally finite permutation is locally finite as well.)
The sets of isomorphism classes of finite and of locally finite $G$-permutations are semirings
with respect to the usual operations: the disjoint union and the Cartesian product.
The set of irreducible (i.e. not representable as the disjoint union of nontrivial elements)
$G$-permutations is in one-to-one correspondence with the set $consubf(\Z\times G)$ of the conjugacy
classes of subgroups of finite index in the group $\Z\times G$.
A subgroup ${\hat H}\subset\Z\times G$ of finite index can be described by a triple 
$(H,m,\alpha)$, where $H$ is a subgroup of the group $G$, $m$ is a natural number, $\alpha$ is an
element of the group $C_H/H$, where $C_H$ is the centralizer of the subgroup $H$:
$C_H=\{a\in G:a^{-1}Ha=H\}$.
A triple $(H,m,\alpha)$ corresponds to the subgroup ${\hat H}(H,m,\alpha)$ of the group $\Z\times G$
generated by $\{0\}\times H$ and $(m,\alpha)$. For $g\in C_H$, we permit ourselves to write
${\hat H}(H,m,g)$ instead of ${\hat H}(H,m,[g])$, where $[g]$ is the class of the element $g$ in $C_H/H$.
Let us denote the conjugacy class of the subgroup ${\hat H}(H,m,\alpha)$ in $\Z\times G$ by
${\hat{\mathfrak{H}}}(H,m,\alpha)$.
The triples $(H,m,\alpha)$ and $(H',m',\alpha')$ are equivalent (i.e. the corresponding subgroups in
$\Z\times G$ are conjugate to each other) if $m'=m$ and there exists $g\in G$ such that $H'=g^{-1}Hg$,
$\alpha'=g^{-1}\alpha g$.
Let $K_0(\mbox{f.}G-\mbox{perm.})$ and $K_0(\mbox{l.f.}G-\mbox{perm.})$ be the Grothendieck rings
of finite and of locally finite $G$-permutations respectively.
As a group, $K_0(\mbox{f.}G-\mbox{perm.})$ is the free Abelian group generated by
the isomorphism classes of the irreducible $G$-permutations;
$$
K_0(\mbox{l.f.}G-\mbox{perm.})=
\left\{\sum\limits_{{\hat{\mathfrak{H}}}\in consubf(\Z\times G)}
k_{\hat{\mathfrak{H}}}[\Z\times G/{\hat{H}}]\right\}\,,
$$
where, in general, one may have infinitely many non-zero coefficients $k_{\hat{\mathfrak{H}}}$
($\hat{H}$ is a representative of the class $\hat{\mathfrak{H}}$). 

For a locally finite $G$-permutation $\sigma$ regarded as a map of a discrete topological space into
itself, there is defined the equivariant Lefschetz number $L^G(\sigma)\in K_0(\mbox{f.}G-\mbox{s.})$
in the sense of \cite{LR}. This is simply the set of fixed points of the permutation $\sigma$
(i.e. of the action of the element $(1,e)\in\Z\times G$), regarded as a $G$-set.
For a virtual $G$-permutation $[\sigma_1]-[\sigma_2]\in K_0(\mbox{l.f.}G-\mbox{perm.})$,
it should be defined as $L^G(\sigma_1)-L^G(\sigma_2)$. 
(For $g\in G$, the map $g\sigma$ (i.e. the action of the element $(1,g)\in\Z\times G$)
is also a locally finite $G$-permutation and therefore the equivariant Lefschetz number $L^G(g\sigma)$
is defined.)

Let $\varphi:X\to X$ be a $G$-equivariant map of a $G$-space $X$ (``good'' in the sense described above)
into itself. 

\begin{definition}
{\em The equivariant zeta function} $\zeta^G_{\varphi}$ of the map $\varphi$ is a virtual
$G$-permutation $[\sigma]\in K_0(\mbox{l.f.}G-\mbox{perm.})$ such that
$L^G(g\sigma^m)=L^G(g{\varphi}^m)$ for all $m\ge 1$ and $g\in G$.
\end{definition}

\begin{proposition}
The equivariant zeta function of a $G$-equivariant map $\varphi$ is well defined
(i.e. the corresponding (virtual) $G$-permutation $[\sigma]$ exists and is unique).
\end{proposition}

\begin{proof}
Assume that
$$
\zeta^G_{\varphi}=
\sum_{{\hat{\mathfrak{H}}}\in consubf(\Z\times G)} k_{\hat{\mathfrak{H}}}[(\Z\times G)/{\hat{H}}]\,.
$$
Then, for $g\in C_H$, the coefficient $\ell_h(m,g)$  at the generator $[G/H]$ in $L^G(g\varphi^m)$ is
equal to
$$
\sum_{{\hat{\mathfrak{H}}}\subset {\hat{\mathfrak{H}}}(H,m,g)} k_{\hat{\mathfrak{H}}}\,.
$$
The system
$$
\ell_{\mathfrak{H}}(m,g)=
\sum_{{\hat{\mathfrak{H}}}\subset {\hat{\mathfrak{H}}}(H,m,g)} k_{\hat{\mathfrak{H}}}
$$
is triangular (with respect to the natural ordering of the triples $(\mathfrak{H},m,\alpha)$)
system of equations, which permits to find the coefficients $k_{\hat{\mathfrak{H}}}$ recurrently.
\end{proof}

The following properties of the equivariant zeta function are easy corollaries from the definition.

\begin{proposition}\label{homotopy}
If $G$-maps $\varphi_i:X\to X$, $i=1,2$, are homotopic in the class of $G$-maps, then
$\zeta^G_{\varphi_1}=\zeta^G_{\varphi_2}$.
\end{proposition}

\begin{proposition}\label{multi}
 Suppose that a $G$-map $\varphi:X\to X$ sends to itself a $G$-invariant subspace
$Y\subset X$ and its complement $X\setminus Y$. Then $\zeta^G_{\varphi}=
 \zeta^G_{\varphi_{\vert Y}}+\zeta^G_{\varphi_{\vert X\setminus Y}}$.
\end{proposition}

\begin{proposition}\label{product}
 For $G$-maps $\varphi_i:X_i\to X_i$, $i=1,2$, one has the equation 
 $\zeta^G_{\varphi_1\times\varphi_2}=\zeta^G_{\varphi_1}\cdot\zeta^G_{\varphi_2}$.
 $($Here $\varphi_1\times\varphi_2$ is the map of the space $X_1\times X_2$ into itself
sending $(x_1, x_2)\in X_1\times X_2$ to $(\varphi(x_1),\varphi(x_2))$.$)$
\end{proposition}

\begin{remark}
 Propositions ~\ref{multi} and \ref{product} permit to define the equivariant monodromy zeta function
 as a homomorphism from the Grothendieck ring of germs of constructible $G$-invariant functions
 on germs of constructible $G$-sets into the Grothendieck ring of finite $(\Z\times G)$-sets.
\end{remark}

\begin{proposition}\label{elementary}
 Let $X=X^{(\mathfrak{H})}$, i.e. the isotropy groups of all points of $X$ are conjugate 
 to a subgroup $H\in {\mathfrak{H}}$,
 and let $\varphi:X\to X$ be a map such that $g\varphi^m(x)\ne x$ for all $x\in X$,
 $g\in G$, $m=1, \ldots, m_0-1$ (i.e., for $1\le m<m_0$, the map $\varphi^m$ does not send any
 $G$-orbit into itself); $g_0\varphi^{m_0}= id$. (Notice that in this case $g_0$ lies in
 the centralizer $C_H$ of the subgroup $H$.) Then
$$
\zeta^G_{\varphi}=\frac{\chi(X/G)}{m_0}[(\Z\times G)/{\hat{H}}(H, m_0, g_0)]
=\chi(X/(\Z\times G))[(\Z\times G)/{\hat{H}}(H, m_0, g_0)]
$$
(in this case the factor $X/(\Z\times G)$ is defined.).
\end{proposition}

Let us formulate equivariant analogues of some well-known statements about the monodromy
and its zeta function.

Let $f_1:(V_1,0)\to(\C,0)$ and $f_2:(V_2,0)\to(\C,0)$ be two germs of $G$-invariant complex
analytic functions.
{\em The Sebastiani--Thom sum} of the germs $f_1$ and $f_2$ is the germ
$f_1\oplus f_2:(V_1\times V_2,0)\to(\C,0)$
given by the equation $(f_1\oplus f_2)(x,y)=f_1(x)+ f_2(y)$. The germ $f_1\oplus f_2$ is invariant
with respect to the natural (diagonal) $G$-action on $(V_1\times V_2,0)$.
The following statement is an equivariant analogue of the Sebastiani--Thom theorem which says 
that the monodromy operator of the Sebastiani--Thom sum of function germs on the affine spaces
$(\C^n,0)$ with isolated critical points at the origin is equal to the tensor product of the monodromy
operators of the germs (or rather of its corollary for the monodromy zeta function).

\begin{theorem} (the equivariant Sebastiani--Thom theorem) 
$$\zeta^G_{f_1\oplus f_2} = \zeta^G_{f_1} + \zeta^G_{f_2} - \zeta^G_{f_1}\cdot\zeta^G_{f_2}\,.$$
\end{theorem}

\begin{proof}
 The statement follows from the fact that the Milnor fibre $V_{f_1\oplus f_2}$ of the germ
$f_1\oplus f_2$ can be $G$-equivariantly retracted (by a deformation) onto the join
$V_{f_1}\star V_{f_2}$ of the Milnor fibres of the germs $f_1$ and $f_2$, see, e.g.
\cite[section 2.7]{AGV1}. Moreover, the monodromy transformation on $V_{f_1}\star V_{f_2}$ is the
join of the corresponding transformations on $V_{f_1}$ and on $V_{f_2}$. Finally, the join
$V_{f_1}\star V_{f_2}$ can be represented as the (disjoint) union of the $G$-invariant subspaces
$V_{f_1}$, $V_{f_2}$ and 
$(V_{f_1}\star V_{f_2})\setminus(V_{f_1}\cup V_{f_2})\simeq V_{f_1}\times V_{f_2}\times (0,1)$
(with the trivial action of $G$ on $(0,1)$), to which one can apply Proposition~\ref{multi}.
\end{proof}

Let $(V, 0)$ be a germ of a complex analytic $G$-space of pure dimension $n$ and let $f:(V, 0)\to(\C,0)$
be a $G$-invariant germ of a complex analytic function such that the set $V_{sing}$ of the singularities
of the space $V$ is contained in the zero-level set $\{f=0\}$ of the germ $f$
(for example, $V$ has an isolated singularity at the origin). Let $\pi:(X,D)\to(V,0)$
($D=\pi{-1}(0)$) be a $G$-resolution of the singularity $f$ which is an isomorphism outside of
$\{f=0\}$. This means that:
\begin{enumerate}
\item[1)] $X$ is a (non-singular) complex analytic manifold of dimension $n$ with an action of the
group $G$;
\item[2)] $\pi$ is a proper analytic map commuting with the action of the group $G$;
\item[3)] the total transform $\pi^{-1}(\{f=0\})$ of the zero-level set of the germ $f$ is a
normal crossing divisor on $X$.
\end{enumerate}
The last condition is equivalent to the fact that, in a neighbourhood of any point 
$x\in\pi^{-1}(\{f=0\})$, there exists a local coordinate system $u_1, \ldots, u_n$ on $X$ centred
at the point $x$, such that the lifting $\widetilde{f}:=f\circ\pi$ of the germ $f$ to the space of
resolution $X$ has the form $\widetilde{f}(u_1, \ldots, u_n)=u_1^{m_1}\cdot\ldots\cdot u_n^{m_n}$,
where $m_i$ are non-negative integers not all equal to zero. We also will assume that
{\em the exceptional divisor} $D = \pi^{-1}(0)$ is a normal crossing divisor on $X$ as well and that
for any point $x\in D$ all the components $E_i$ of the exceptional divisor $D$ passing through the
point $x$ are invariant with respect to the isotropy subgroup $G_x$ of the point $x$.

\begin{remark}
 The last condition is not too essential and is required here for convenience of reasonings.
 One can always achieve this by additional blow-ups along pairwise intersections of the components $E_i$.
 There appear new components whose non-singular parts (in the exceptional divisor) have zero Euler
 characteristics. It is not difficult to see that they don't contribute to the right hand side 
 of the equation~(\ref{ACampo}) below.
\end{remark}

Let $D=\bigcup_{i=1}^s E_i$ be the decomposition of the exceptional divisor $D$ into the irreducible
components and let ${\stackrel{\circ}{E}}_i$ be the nonsingular part of the component $E_i$ in the
total transform $\pi^{-1}(\{f=0\})$ of the zero-level set, i.e. the component $E_i$ itself minus 
its intersections with all other components of $\pi^{-1}(\{f=0\})$.
Let ${\stackrel{\circ}{D}}=\bigcup_{i=1}^s {\stackrel{\circ}{E}}_i$ be the nonsingular part
of the exceptional divisor $D$ in $\pi^{-1}(\{f=0\})$, let $\hat D={\stackrel{\circ}{D}}/G$
be the corresponding factor space, i.e. the space of orbits of the action of the group $G$ on
${\stackrel{\circ}{D}}$, and let $j:{\stackrel{\circ}{D}}\to \hat D$ be the factorization map. 

Let $\{\Xi\}$ be a stratification of the space $\hat D$ such that:
\begin{enumerate}
\item[1)] each stratum $\Xi$ is a connected complex analytic manifold and therefore
it is contained in the image of one of the components ${\stackrel{\circ}{E}}_i$ under 
the action of the map $j$;
\item[2)] over each stratum $\Xi$ the map $j$ is a covering.
\end{enumerate}

As above, for a point $x\in X$, let $G_x$ be its isotropy subgroup. The property 2) implies that the
isotropy subgroups of all the points of a connected component of the preimage $j^{-1}(\Xi)$ of a stratum
$\Xi$ coincide and the isotropy subgroups of all the points of the preimage $j^{-1}(\Xi)$ are conjugate
to each other in $G$. Let us denote the corresponding conjugacy class by ${\mathfrak{H}}_{\Xi}$.
For a stratum $\Xi$, let us take a point $x\in j^{-1}(\Xi)$. Let $m_x=m_{\Xi}$ be the multiplicity
of zero of the lifting $\widetilde{f}=f\circ\pi$ of the function $f$ to the space $X$ of the resolution
along the component of ${\stackrel{\circ}{D}}$ containing the point $x$. (It is also the multiplicity of
the function $\widetilde{f}$ at the point $x$.) The isotropy subgroup $G_x$ of the point $x$ acts on
the fibre over the point $x$ of the (one-dimensional) normal fibration to the non-singular part
${\stackrel{\circ}{D}}$ of the exceptional divisor. Let $H_x$ be the subgroup of the group $G_x$
consisting of the elements whose action on the fibre is trivial. Notice that $G_x\subset C_{H_x}$,
the factor $G_x/H_x$ is a cyclic group whose order $n_x=\vert G_x/H_x\vert=:n_{\Xi}$ divides the
multiplicity $m_x$. The action (a representation) of the group $G_x/H_x$ on the fibre defines a
homomorphism (even a monomorphism) of the group $G_x/H_x$ to $\C^{\star}={\rm GL}(1,\C)$. Let
$\alpha_x^{-1}$ be the element of this group which maps to $\exp\left({2\pi i}/{n_x}\right)$.
Let ${\hat{H}}_x={\hat{H}}(H_x, \frac{m_x}{n_x}, \alpha_x)$. It is not difficult to see that the
conjugacy class ${\hat{\mathfrak{H}}}(H_x, \frac{m_x}{n_x}, \alpha_x)$ of the subgroup
${\hat{H}}(H_x, \frac{m_x}{n_x}, \alpha_x)$ in $\Z\times G$ does not depend on the choice of a
point $x\in j^{-1}(\Xi)$. A representative of this class we shall denote by ${\hat{H}}_{\Xi}$.

\begin{theorem}\label{teoACampo} (the equivariant A'Campo formula)
\begin{equation}\label{ACampo}
 \zeta^G_f=\sum_{\Xi}\chi(\Xi)\cdot[(\Z\times G)/{\hat{H}}_{\Xi}]\,.
\end{equation}
\end{theorem}

\begin{proof}
Since the resolution $\pi$ is an isomorphism outside of $f^{-1}(0)$, the map $\pi$ defines an
identification of the Milnor fibre $M_f$ with its preimage $\pi^{-1}(M_f)\subset X$. In the same way
as in the usual (non-equivariant) situation (see~\cite{Clemens, A'Campo}) one can construct a
$G$-equivariant deformation retraction of a neighbourhood of the total transform
$\pi^{-1}(f^{-1}(0))={\widetilde{f}}^{-1}(0)$ of the zero-level set $f^{-1}(0)$ onto the total
transform ${\widetilde{f}}^{-1}(0)$ itself such that, outside of a neighbourhood of the union
of all the intersections of the components of ${\widetilde{f}}^{-1}(0)$, i.e. where in local
coordinates ${\widetilde{f}}(u_1,\ldots, u_n)=u_1^m$, the retraction sends the point
$(u_1,u_2, \ldots, u_n)$ to the point $(0, u_2,\ldots, u_n)$ (in some coordinates).
One can also construct a $G$-equivariant monodromy transformation $\widetilde{\Gamma}$ on
$\pi^{-1}(M_f)$ commuting with the retraction and such that, in a neighbourhood of a point, where
${\widetilde{f}}(u_1,\ldots, u_m)=u_1^m$, it sends $(u_1,u_2, \ldots, u_n)\in \pi^{-1}(M_f)$ to
$(\exp{(2\pi i/m)}u_1, u_2,\ldots, u_n)$.
From Proposition~\ref{multi}, it follows that the equivariant zeta function
$\zeta^G_f=\zeta^G_{\widetilde{\Gamma}}$ is equal to the sum of the equivariant zeta functions
of the monodromy transformation on the preimages of the strata $\Xi$ and on neighbourhoods of the
intersections of the components of the total transform ${\widetilde{f}}^{-1}(0)$.

It is not difficult to show that the equivariant zeta function of the monodromy transformation on
a neighbourhood of an intersection of the components of ${\widetilde{f}}^{-1}(0)$ is equal to zero.
This can be deduced from the fact that, in such a neighbourhood, the Milnor fibre can be represented
as a fibration by circles coordinated with the group action and the monodromy transformation can be
constructed such that it preserves this fibration. 
 
The equivariant zeta function of the monodromy transformation $\widetilde{\Gamma}$ on the preimage
of a stratum $\Xi$ is given by Proposition~\ref{elementary}. Indeed, the intersection of the Milnor
fibre $\pi^{-1}(M_f)$ with the normal fibre to ${\stackrel{\circ}{D}}$ at a point from $\pi^{-1}(\Xi)$
consists of $m_{\Xi}$ points permuted by the monodromy transformation $\widetilde{\Gamma}$ cyclically.
The intersections with the orbits of the group $G$ divide these points into $n_{\Xi}$ groups with
$m_{\Xi}/n_{\Xi}$ points in each. (Here, of course, the term ``group'' does not have the same meaning as
at other places, but, to my surprise, I did not find an appropriate synonym.)
The monodromy transformation $\widetilde{\Gamma}$ permutes these groups cyclically. Therefore, for
$m=1, \ldots, (\frac{m_{\Xi}}{n_{\Xi}}-1)$, the $m$-th power of the monodromy transformation does not
preserve any $G$-orbit. In its turn the $\frac{m_{\Xi}}{n_{\Xi}}$-th power of the monodromy
transformation sends each point $y$ from the normal fibre at the point $x$ to the point $a^{-1}y$,
where $a\in\alpha_x$. Therefore the transformation $a\circ{\widetilde{\Gamma}}^{\frac{m_{\Xi}}{n_{\Xi}}}$
acts on the intersections of the Milnor fibre $\pi^{-1}(M_f)$ with the normal fibres to
${\stackrel{\circ}{D}}$ at all the point of $\pi^{-1}(\Xi)$ identically, what is required in
Proposition~\ref{elementary}. Therefore the equivariant zeta function under consideration is equal to
$\chi(\Xi)\cdot[(\Z\times G)/{\hat{H}}_{\Xi}]$, what proves the statement.
\end{proof}


\end{document}